\documentclass{amsart}

\usepackage{amssymb}
\usepackage{color}
\usepackage{amsxtra} 
\usepackage{mathtools} 

\newtheorem{theorem}{Theorem}
\newtheorem{lem}[theorem]{Lemma}
\newtheorem{prop}[theorem]{Proposition}
 
\newtheorem*{theorem*}{Theorem} 

\theoremstyle{definition}

\theoremstyle{remark}
\newtheorem{remark}[theorem]{Remark}

\numberwithin{equation}{section}

\begin{document}

\title[Ill-posedness of the incompressible Euler equations in the $C^1$ space] {Local ill-posedness of the incompressible Euler equations in the $C^1$ space}

\author{Gerard Misio\l ek}
\address{Department of Mathematics, University of Colorado, Boulder, CO 80309-0395, USA 
and 
Department of Mathematics, University of Notre Dame, IN 46556, USA} 
\email{gmisiole@nd.edu} 

\author{Tsuyoshi Yoneda}
\address{Department of Mathematics, Tokyo Institute of Technology, Meguro-ku, Tokyo 152-8551, Japan} 
\email{yoneda@math.titech.ac.jp}
\thanks{The second author was partially supported by JSPS KAKENHI Grant Number 25870004.} 

\subjclass[2000]{Primary 35Q35; Secondary 35B30}

\date{\today} 


\keywords{Euler equations, ill-posedness, Lagrangian flow, norm inflation} 

\begin{abstract}
We prove that the 2D Euler equations are not locally well-posed in $C^1$. 
Our approach relies on the technique of Lagrangian deformations and norm inflation 
of Bourgain and Li. We show that the assumption that the data-to-solution map 
is continuous in $C^1$ leads to a contradiction with a well-posedness result in $W^{1,p}$ 
of Kato and Ponce. 
\end{abstract}

\maketitle

\section{Introduction} 
\label{sec:Intro} 

There is an extensive literature dedicated to well-posedness of the Cauchy problem 
for the incompressible Euler equations of hydrodynamics. 
The first rigorous results on the local in time existence and uniqueness of solutions 
go back to the papers of Gyunter \cite{Gu} and Lichtenstein \cite{Li} in the late 1920's 
while the first global result was proved in 2D by Wolibner \cite{Wo} in 1933. 
Nevertheless, our understanding of the Cauchy problem remains incomplete 
especially in connection with the phenomenon of turbulence and persistence of smooth solutions in 3D 
for all time. 

Another important problem is to identify an optimal function space in which the Cauchy problem 
is locally well-posed. In this area substantial progress has been made in recent years. 
For example, Bardos and Titi \cite{BT} used the shear flow of DiPerna and Majda 
to construct solutions in 3D with an instantaneous loss of regularity in H\"{o}lder $C^\alpha$ 
and Zygmund $B^1_{\infty, \infty}$ spaces. More precisely, they found $C^\alpha$ initial data 
for which the corresponding (weak) solution does not belong to $C^\beta$ for any $1>\beta > \alpha^2$ 
and any $t>0$. This technique has also been used to obtain similar results 
in the Triebel-Lizorkin $F^1_{\infty, 2}$ space by Bardos, Lemarie and Titi 
and in the logarithmic Lipschitz spaces $\mathrm{logLip}^\alpha$ by the authors \cite{MY}.
More recently, Bourgain and Li \cite{BL} using a combination of Lagrangian and Eulerian techniques 
obtained strong local ill-posedness results in the Sobolev spaces $W^{n/p+1,p}$ for any $1<p<\infty$ 
and in the Besov spaces $B^{n/p+1}_{p,q}$ with $1<p<\infty$ and $1<q\leq\infty$ and $n=2$ or $3$. 
In particular, they settled the borderline Sobolev case $H^{n/2+1}$. 

However, as far as we are aware the problem of local well-posedness in the classical $C^1$ space 
in both space dimensions as well as other spaces such as $B^1_{\infty, q}$ with $1< q < \infty$ 
has remained open; cf. comments on criticality of the $C^1$ space in \cite{BT}; 
see also the papers of Pak and Park \cite{PP} and Takada \cite{Ta}. 
Our goal in this paper is to settle the former case in 2D by showing that the Euler equations are 
locally ill-posed in $C^1(\mathbb{R}^2)$. 

Recall that a Cauchy problem is locally well-posed in a Banach space $X$ (in the sense of Hadamard) 
if for any initial data in $X$ there exist $T>0$ and a unique solution which persists in the space $C([0,T), X)$ 
and which depends continuously on the data. Otherwise, the problem is said to be ill-posed. 
The Cauchy problem for the Euler equations in 2D is usually written in the form 
\begin{align} \label{eq:Euler-u} 
&u_t + u{\cdot}\nabla u + \nabla\pi = 0, 
\qquad 
t \geq 0, \, x \in \mathbb{R}^2 
\\ \label{eq:Euler-uu} 
&\mathrm{div}\, u = 0 
\\ \label{eq:Euler-u-ic} 
&u(0) = u_0, 
\end{align} 
where $u$ is the velocity vector field and $\pi$ is the pressure function of the fluid. 
Our approach is inspired by the methods of Bourgain and Li \cite{BL} who for suitable initial vorticity data 
constructed Lagrangian flows with large deformation gradients and used them to show that there exist 
nearby solutions which lose their regularity instantaneously in time through norm inflation. 
The initial data has an odd symmetry and a stagnation point at the origin. Such properties also seem to play 
an important role in a paper of Kiselev and \v{S}verak \cite{KS}. 
Additional information about other recent ill-posedness results can be found in both of these references. 

We mention in passing yet another manifestation of local ill-posedness 
that occurs for the Euler as well as the (supercritical) quasi-geostrophic equations 
in which certain initial data defined in the periodic case by lacunary series 
lead to solutions that fail to be continuous in time when considered as curves 
in the classical H\"{o}lder $C^{1+\alpha}$ spaces for $0<\alpha<1$; 
we refer to Cheskidov and Shvydkoy \cite{CS} and \cite{MY} for details. 

The main result of the paper can be succinctly stated as follows 
\begin{theorem} \label{thm:0} 
The 2D incompressible Euler equations \eqref{eq:Euler-u} are locally ill-posed in the space $C^1$. 
\end{theorem} 

Before giving a more precise statement it will be convenient to use the vorticity formulation of 
the Euler equations. Recall that in two dimensions the vorticity of a vector field $u$ is a 2-form 
$\omega = d u^\flat$ which is identified with the function 
$$ 
\omega = \mathrm{rot}\, u = - \frac{\partial u_1}{\partial x_2} + \frac{\partial u_2}{\partial x_1}. 
$$ 
In this case the Cauchy problem \eqref{eq:Euler-u}-\eqref{eq:Euler-u-ic} can be rewritten as 
\begin{align} \label{eq:euler-v} 
&\omega_t + u{\cdot}\nabla \omega = 0, 
\qquad 
t \geq 0, \; x \in \mathbb{R}^2 
\\  \label{eq:euler-vic} 
&\omega(0) = \omega_0 
\end{align} 
where the velocity is recovered from $\omega$ using the Biot-Savart law 
\begin{equation} \label{eq:Biot-Savart} 
u = K \ast \omega = \nabla^\perp \Delta^{-1} \omega
\end{equation} 
with kernel $K(x) {=} (2\pi)^{-1}(-x_2/|x|^2, x_1/|x|^2)$ and where 
$\nabla^\perp {=} (-\frac{\partial}{\partial x_2}, \frac{\partial}{\partial x_1})$ 
denotes the symplectic gradient of a function. 

Our strategy will be the following. 
First, as in \cite{BL} we choose an initial vorticity $\omega_0$ such that 
the Lagrangian flow of the corresponding velocity field retains a large gradient 
on a (possibly short) time interval. 
We then perturb $\omega_0$ to get a sequence of initial vorticities in $W^{1,p}$. 
Finally, we show that 
\textit{the assumption that the Euler equations are well-posed in} $C^1(\mathbb{R}^2)$ 
(in particular, that its solutions depend continuously on the initial data in the $C^1$ norm) 
leads to a contradiction with a result of Kato and Ponce, which for convenience we restate 
in the following form 
\begin{theorem*}[Kato-Ponce \cite{KP}] 
Let $1{<}p{<}\infty$ and $s{>}1{+}\frac{2}{p}$. For any $\omega_0 \in W^{s-1, p}(\mathbb{R}^2)$ 
and any $T>0$ there exists a constant $K=K(T,\omega,s,p)>0$ such that 
$$
\sup_{0 \leq t \leq T}\| \omega(t)\|_{W^{s-1,p}} \leq K. 
$$ 
\end{theorem*} 
Theorem \ref{thm:0} will be therefore a consequence of the following result
\begin{theorem} \label{thm:1} 
Let $2 < p < \infty$. 
There exist $T>0$ and a sequence $\omega_{0,n} \in C^\infty_c(\mathbb{R}^2)$ 
with the following properties 

1. there exists a constant $C>0$ such that 
$\| \omega_{0,n} \|_{W^{1.p}} \leq C$ for all $n \in \mathbb{Z}_+$ 

\noindent and 

2. for any $M \gg 1$ there is $0 < t_0 \leq T$ such that 
$\| \omega_{n}(t_0)\|_{W^{1,p}} \geq M^{1/3}$ 
for all sufficiently large $n$ and all $p$ sufficiently close to $2$. 
\end{theorem} 
In Section \ref{sec:Lag} we provide some technical lemmas to construct an initial vorticity 
whose Lagrangian flow has a large gradient. 
The proof of the latter is given in Section \ref{sec:proof-Prop}. 
The last section contains the proof of Theorems \ref{thm:1}.

\section{Vorticity and the Lagrangian flow} 
\label{sec:Lag} 

Given a smooth radial bump function $\varphi$ on $\mathbb{R}^2$ supported in the unit ball $B(0,1)$ 
with $0 \leq \varphi \leq 1$ define another function 
\begin{align} \label{eq:bump} 
\varphi_0(x_1, x_2) 
= 
\sum_{\varepsilon_1, \varepsilon_2 = \pm 1} 
\varepsilon_1 \varepsilon_2 \varphi(x_1 {-} \varepsilon_1, x_2 {-} \varepsilon_2). 
\end{align} 
For a fixed positive integer $N_0 \in \mathbb{Z}_+$ and any $M \gg 1$ we set 
\begin{align} \label{eq:iv} 
\omega_0(x) = \omega_0^{M,N}(x) 
= 
M^{-2} N^{-\frac{1}{p}} \sum_{N_0 \leq k \leq N_0+N} \varphi_k(x), 
\qquad 
N = 1, 2, 3 \dots 
\end{align} 
where $2< p < \infty$ and where 
\begin{align*} 
\varphi_k(x) = 2^{(-1 + \frac{2}{p})k} \varphi_0 (2^k x). 
\end{align*} 
Observe that by construction $\varphi_0$ is an odd function in both $x_1, x_2$ and 
for any $k \geq 1$ its support is compact and contained in the set 
\begin{equation} \label{eq:suppp} 
\mathrm{supp}\,{ \varphi_k } 
\subset 
\bigcup_{\varepsilon_1, \varepsilon_2 = \pm 1} 
B\big( (\varepsilon_1 2^{-k}, \varepsilon_2 2^{-k}), 2^{-(k+2)} \big). 
\end{equation} 
Combined with the uniform (in time) $L^\infty$ control of the vorticity in $\mathbb{R}^2$ 
this ensures the existence of a unique solution of the Cauchy problem \eqref{eq:euler-v}-\eqref{eq:euler-vic} 
with the initial data \eqref{eq:iv}; e.g., by a result of Yudovich \cite{Yu}, see also Majda and Bertozzi \cite{MB}. 
Moreover 
\begin{lem} \label{lem:omega-0} 
We have 
\begin{align} \label{eq:omega-0} 
\| \omega_0 \|_{L^p} + \| \omega_0 \|_{\dot{W}^{1,p}} \lesssim  M^{-2} 
\end{align} 
with the bound independent of $N>0$ and $2<p<\infty$. 
\end{lem} 
\begin{proof} 
Since the supports in \eqref{eq:suppp} are disjoint  we have 
\begin{align*} 
\| \omega_0 \|_{L^p}^p 
=
M^{-2p} N^{-1} \sum_{N_0 \leq k \leq N_0 +N} 2^{-kp} \int_{\mathbb{R}^2} \big| \varphi_0(x) \big|^p dx 
\lesssim 
M^{-2p} 
\end{align*} 
and similarly 
\begin{align*} 
\Big\| \frac{\partial \omega_0}{\partial x_1} \Big\|_{L^p}^p 
= 
M^{-2p} N^{-1} \sum_{N_0 \leq k \leq N_0 +N} 
\int_{\mathbb{R}^2} 2^{2k} \Big| \frac{\partial\varphi_0}{\partial x_1} (2^kx) \Big|^p dx. 
\simeq 
M^{-2p} 
\end{align*} 
The estimate of the other partial is analogous. 
\end{proof} 
In particular, since $p>n=2$ from the results of Kato and Ponce it follows that 
there exists a unique velocity field $u \in C^1([0, \infty), W^{2,p}(\mathbb{R}^2))$ 
solving \eqref{eq:Euler-u}-\eqref{eq:Euler-uu} 
whose vorticity function 
$\omega \in C([0,\infty), W^{1,p}(\mathbb{R}^2))$ 
satisfies the initial condition \eqref{eq:iv} 
(see \cite{KP}, Lem. 3.1; Thm. III). 

The associated Lagrangian flow of $u = \nabla^\perp\Delta^{-1}\omega$, 
i.e., the solution of the initial value problem 
\begin{align} \label{eq:flow} 
&\frac{d}{dt}\eta(t,x) = u(t, \eta(t,x)) \; \big( {=} F_u(\eta(t,x)) \big) 
\\  \label{eq:flow-ic} 
&\eta(0,x) = x 
\end{align} 
is a curve of volume-preserving diffeomorphisms with 
$\omega\circ\eta \in C([0, \infty), W^{1,p}(\mathbb{R}^2))$, 
see e.g., \cite{KP} or \cite{BB}. 
Furthermore, the odd symmetry of $\omega_0$ is preserved by $\eta$ and hence 
retained by the vorticity $\omega$ for all time. 
From \eqref{eq:Biot-Savart} it then follows that the velocity field $v$ is symmetric 
with respect to the variables $x_1$ and $x_2$ and hence both coordinate axes 
are invariant under the flow $\eta$ with the origin $x_1=x_2=0$ as its hyperbolic stagnation point. 

\begin{lem} \label{lem:Riesz} 
Let $T>0$ and consider the flow $\eta(t)$ of the velocity field $u=\nabla^\perp\Delta^{-1}\omega$ 
for $0 \leq t \leq T$. 
Suppose that $\sup_{0\leq t \leq T} \| D\eta(t)\|_\infty \leq C_T$ for some $C_T>0$. Then 
\begin{equation} \label{eq:Riesz} 
\sup_{0 \leq t \leq T} \| R_{ii} \omega (t) \|_\infty 
\lesssim 
\big( 5/4 + T C_T \big)^{\frac{p-2}{p}} C_T M^{-2} 
\end{equation} 
where $R_{ij} = \partial_i \partial_j \Delta^{-1}$ denotes 
the double Riesz transforms with $i, j = 1,2$. 
\end{lem} 
\begin{proof} 
Observe that $\mathrm{supp}\, \omega_0 \subset B(0, 5/4)$ by \eqref{eq:iv} and \eqref{eq:suppp}. 
An estimate of the Biot-Savart operator \eqref{eq:Biot-Savart} gives a uniform bound on the velocity field 
so that the support of the vorticity can grow at most linearly in time and, 
since 
$\omega = \omega_0\circ\eta^{-1}$ 
by conservation of vorticity, we find that the support of $\omega(t)$ is contained in a ball of 
radius $r_t = 5/4 + tC_T$. 

Next, using the H\"older inequality we obtain 
\begin{align*} 
\sup_{x \in B(0,r_T)} | R_{ii} \omega(t,x) | 
&= 
\bigg| \frac{1}{2\pi} \int_{B(0,r_T)} \frac{x_i - y_i}{|x-y|^2} \frac{\partial \omega}{\partial x_i}(t,y) \, dy \bigg|  
\\ 
&\leq 
\frac{1}{2\pi}  \left( \int_{B(0,r_T)} |x-y|^{-q} dy \right)^{1/q}  \| \nabla \omega(t) \|_{L^p} 
\\ 
&\lesssim 
r_T^{\frac{2-q}{q}} \big\| D\eta^{-1}(t) \big\|_\infty \big\| \nabla\omega_0 \circ \eta^{-1}(t) \big\|_{L^p} 
\end{align*} 
where $1/p + 1/q =1$. 
Since $D\eta^{-1} = (D\eta)^{-1} \circ \eta^{-1}$ and $\eta(t)$ is volume-preserving,  
using the bound on the Jacobi matrix of the flow and inequality \eqref{eq:omega-0} of Lemma \ref{lem:omega-0} 
we can further estimate the expression on the right hand side by 
\begin{align*} 
\simeq 
r_T^{\frac{p-2}{p}} \| D\eta(t) \|_\infty  \| \nabla\omega_0 \|_{L^p} 
\lesssim 
r_T^{\frac{p-2}{p}} C_T M^{-2} 
\end{align*} 
which gives \eqref{eq:Riesz}. 
\end{proof} 
\begin{remark} \label{rem:SGr} 
In fact, note that if $\xi: \mathbb{R}^2 \to \mathbb{R}^2$ is a volume-preserving diffeomorphism 
then the Jacobi matrix of its inverse can be computed from 
\begin{align*} 
D\xi^{-1} 
= 
(D\xi)^{-1}\circ\xi^{-1} 
= 
\left( 
\begin{matrix} 
\partial_2\xi_2 {\circ} \xi^{-1} & -\partial_2 \xi_1{\circ} \xi^{-1} 
\\ 
-\partial_1 \xi_2{\circ}\xi^{-1} & \partial_1 \xi_1{\circ} \xi^{-1} 
\end{matrix} 
\right). 
\end{align*} 
Thus given a smooth function $f: \mathbb{R}^2 \to \mathbb{R}$ we can express the gradient 
of the composition $\nabla( f \circ \xi^{-1}) = \nabla{f} \circ \xi^{-1} {\cdot} D\xi^{-1}$ using the scalar product 
and the symplectic gradient as 
\begin{equation} \label{eq:SGr} 
\nabla( f \circ \xi^{-1}) 
=  
\big( 
{-}\nabla{f} \circ \xi^{-1} \cdot \nabla^\perp{\xi_2} \circ \xi^{-1}, 
\nabla{f} \circ \xi^{-1} \cdot \nabla^\perp{\xi_1} \circ \xi^{-1} 
\big). 
\end{equation} 
\end{remark} 

The proof of the following result will be given in the next section. 
\begin{prop} \label{prop:Lag} 
Let $\eta(t)$ be the flow of the velocity field $u=\nabla^\perp\Delta^{-1}\omega$ with initial vorticity 
given by \eqref{eq:iv}. Given $M \gg 1$ we have 
$$ 
\sup_{0 \leq t \leq M^{-3}} \| D\eta(t) \|_\infty > M 
$$ 
for any sufficiently large integer $N>0$ in \eqref{eq:iv} and any $2<p<\infty$ 
sufficiently close to $2$. 
\end{prop} 
\begin{remark} 
In what follows it can be assumed that $2<p\leq3$. 
In this case all estimates on the flow $\eta$ or $D\eta$ can be made independent 
of the Lebesgue exponent $2<p < \infty$. 
\end{remark} 

We will also need a comparison result for solutions of the Lagrangian flow equations, 
namely 
\begin{lem} \label{lem:comp} 
Let $u$ and $v$ be smooth divergence-free vector fields on $\mathbb{R}^2$ and let 
$\eta$ and $\xi$ be the solutions of \eqref{eq:flow}-\eqref{eq:flow-ic} with the right-hand sides 
given by $F_u$ and $F_{u+v}$ respectively. 
Then 
$$ 
\sup_{0 \leq t \leq 1}{ \big( \| \xi(t) - \eta(t) \|_\infty + \| D\xi(t) - D\eta(t) \|_\infty \big) } 
\leq 
C\sup_{0 \leq t \leq 1} ( \| v(t) \|_\infty + \| Dv(t)\|_\infty )  
$$ 
for some $C>0$ depending only on $u$ and its derivatives. 
\end{lem} 
For a standard proof one writes down the equation for the difference $\eta - \xi$ 
and applies Gronwall's inequality, see e.g., \cite{BL}; Lemma 4.1.

\section{Proof of Proposition \ref{prop:Lag}} 
\label{sec:proof-Prop} 

Let $T \leq M^{-3}$ and 
assume to the contrary that 
\begin{equation} \label{eq:D<} 
\| D\eta(t) \|_\infty \leq M 
\end{equation} 
for any $0 \leq t \leq T$. 
Since $D\eta(0) = Id$ shrinking the time interval $[0,T]$ slightly further, if necessary, we can arrange 
by continuity to have $C_T \leq M$, see Lemma \ref{lem:Riesz}. 
Then \eqref{eq:Riesz} gives 
\begin{align} \label{eq:Rii} 
\sup_{0\leq t \leq T}{\| R_{ii}\omega(t)\|_\infty} 
\lesssim 
(5/4 + M^{-3} M)^{(p-2)/p} M M^{-2} 
\simeq 
M^{-1} 
\end{align} 
for $i=1,2$. Note that since $M \gg 1$ the factor in the parenthesis can be bounded 
by a universal constant (for example by 3) 
and so the bound in \eqref{eq:Rii} is independent of any Lebesgue exponent $p>2$.  

Differentiating the flow equations \eqref{eq:flow} in the $x$ variable 
we obtain the system 
\begin{align*} 
\frac{d}{dt} D\eta(t, x) 
&= 
\left( 
\begin{matrix} 
-R_{12} \omega(t, x) & -R_{22} \omega(t, x) 
\\ 
R_{11} \omega(t, x) & R_{12} \omega(t, x) 
\end{matrix} 
\right) 
D\eta(t, x) 
\\ 
&= 
\left( 
\begin{matrix} 
-\Lambda(t,x) & 0 
\\ 
0 & \Lambda(t,x) 
\end{matrix} 
\right) 
D\eta(t, x) 
+ 
P(t, x) D\eta(t, x) 
\\ 
D\eta(0,x) &= Id 
\end{align*} 
where $\Lambda(t, x) = (R_{12}\omega)(t, \eta(t, x))$. 
Observe that by \eqref{eq:Rii} we have 
\begin{equation} \label{eq:P} 
\sup_{0 \leq t \leq T} \| P(t) \|_\infty \lesssim M^{-1}. 
\end{equation} 
Applying Duhamel's formula we can rewrite the above system in the form 
\begin{align} \nonumber 
D\eta(t,x) 
&= 
\left( 
\begin{matrix} 
e^{-\int_0^t \Lambda(\tau,x) d\tau}  &  0 
\\ 
0  &  e^{\int_0^t \Lambda(\tau,x) d\tau} 
\end{matrix} 
\right) 
\\  \label{eq:Duhamel} 
&\hskip 2cm + 
\int_0^t 
\left( 
\begin{matrix} 
e^{-\int_\tau^t \Lambda(\sigma,x) d\sigma}  &  0 
\\ 
0  &  e^{\int_\tau^t \Lambda(\sigma,x) d\sigma} 
\end{matrix} 
\right) 
P(\tau,x) D\eta(\tau,x) \, d\tau 
\\  \nonumber 
&= 
A(t,x) + B(t,x). 
\end{align} 
For any $x \in \mathbb{R}^2$ and any $0 \leq \tau \leq t$ we have the inequalities 
\begin{align} \nonumber 
e^{\mp \int_\tau^t \Lambda(\sigma,x) d\sigma} 
&= 
e^{\mp (\int_0^t \Lambda(\sigma,x) d\sigma - \int_0^\tau \Lambda(\sigma,x) d\sigma) } 
\\  \label{eq:sup-e} 
&\leq 
e^{2\sup_{0 \leq \tau \leq t} | \int_0^\tau \Lambda(\sigma,x) d\sigma |} 
= 
\sup_{0 \leq \tau \leq t} e^{2 | \int_0^\tau \Lambda(\sigma,x) d\sigma |} 
\end{align} 
and therefore, solving for $A(t,x)$ on the right hand side of equation \eqref{eq:Duhamel} 
and using the bounds \eqref{eq:D<} and \eqref{eq:P} we find 
\begin{align*} 
e^{| \int_0^t \Lambda(\tau,x) d\tau |} 
\lesssim 
M + t \sup_{0 \leq \tau \leq t}{ e^{2|\int_0^\tau \Lambda(\sigma,x)d\sigma|} } 
\end{align*} 
for any $0 \leq t \leq T$. 
Recall that $\partial_i \partial_j \Delta^{-1}$ is a Calderon-Zygmund operator mapping continuously 
in $W^{1,p}$ for any $1<p<\infty$ and $\omega\circ\eta \in C([0,\infty), W^{1,p}(\mathbb{R}^2))$. 
Consequently, a simple continuity argument gives 
\begin{equation} \label{eq:e} 
e^{|\int_0^t \Lambda(\tau, x) d\tau|} 
\lesssim 
2M 
\end{equation} 
and hence 
\begin{equation} \label{eq:2M} 
\left| \int_0^t \Lambda(\tau, x) d\tau \right| 
\lesssim 
\log{(2M)} 
\end{equation} 
for any $x \in \mathbb{R}^2$ and any $0 \leq t \leq T \leq M^{-3}$ provided that $M$ is chosen sufficiently large. 
Observe that this bound is independent of the choice of the integers $N$ and $N_0$ in \eqref{eq:iv} 
as well as the exponent $p$. 
In particular, we have 
\begin{equation} \label{eq:e2M} 
(2M)^{-1} 
\lesssim 
e^{\pm \int_0^t \Lambda(\tau, x) d\tau} 
\lesssim 
2M, 
\qquad\quad 
0 \leq t \leq T, 
\quad x \in \mathbb{R}^2. 
\end{equation} 

We will seek a contradiction with \eqref{eq:2M} and to this end we will need to examine 
the expression for $\Lambda(t,0)$. 
First, using \eqref{eq:Duhamel} we have 
\begin{align} \label{eq:fi} 
\eta(t,x) 
&= 
\eta(t,x) - \eta(t,0) 
= 
\int_0^1 D\eta (t, rx){\cdot}x \, dr 
\\ \nonumber 
&= 
\int_0^1 A(t, rx) \, dr \cdot x + \int_0^1 B(t, rx) \, dr \cdot x 
= 
\tilde{A}(t,x) + \tilde{B}(t,x) 
\end{align} 
where 
\begin{align} \nonumber 
| \tilde{B}(t,x) | 
&\leq 
|x| \int_0^1 | B(t, rx)| dr 
\\ \label{eq:beta} 
&\leq 
|x| \int_0^1 \int_0^t 
\sup_{0 \leq \tau \leq t} e^{2|\int_0^\tau \Lambda(\sigma, rx) d\sigma|} \|P(\tau)\|_\infty \|D\eta(\tau)\|_\infty 
d\tau dr 
\\  \nonumber
&\lesssim 
t M^2 |x| 
\lesssim 
M^{-1} |x| 
\end{align} 
by \eqref{eq:D<}, \eqref{eq:P}, \eqref{eq:sup-e} and \eqref{eq:e} 
for any $0 \leq t \leq T$ 
and similarly 
\begin{equation} \label{eq:alfa} 
| \tilde{A}_i (t,x) | 
\leq 
\int_0^1 \Big| \sum_{j=1}^2 A_{ij}(t,rx) x_j \Big| \, dr 
\lesssim 
M |x_i| 
\qquad 
(i =1,2) 
\end{equation} 
where $A_{ij}(t,x)$ denote the entries of the matrix $A(t,x)$ in \eqref{eq:Duhamel}. 

Next, it is not difficult to check that by construction the components $\eta_1$ and $\eta_2$ 
of the Lagrangian flow of $u = \nabla^\perp \Delta^{-1}\omega$ with initial vorticity 
$\omega_0$ given by \eqref{eq:iv} are sign-preserving in the sense that 
$x_i \geq 0$ implies that $\eta_i (x) \geq 0$ for $i =1,2$. 
In fact, a quick inspection shows that 
$ 
(\Delta^{-1} \partial_2 \omega)(t, 0, \eta_2(t,x)) 
= 
(\Delta^{-1} \partial_2 \omega)(t, \eta_1(t,x),0) 
= 0 
$ 
so that both components satisfy an O.D.E. 
$$ 
\frac{d}{dt} \eta_i(t,x) = F_i \big( t,\eta(t,x) \big) \eta_i(t,x) 
\qquad 
(i=1,2) 
$$ 
with some smooth functions $F_i$ and therefore 
$
\eta_i(t,x) 
= 
x_i e^{\int_0^t F_i(\tau, \eta(\tau, x)) d\tau} 
$ 
which implies the assertion. 

Combining this observation with the odd symmetry of $\omega_0$ as in \cite{BL} we find that the integrand 
of the expression for $\Lambda(t,0)$ is a non-negative function in $t$ 
and can be bounded below by its restriction to a subset of the first quadrant. 
Since the origin is a stagnation point of the flow we have $\eta(t,0)=0$ and using conservation of vorticity 
and change of variables we get 
\begin{align} \nonumber 
- \Lambda(t,0) 
&= 
-(R_{12}\omega) (t,\eta(t,0)) 
= 
-\frac{\partial}{\partial x_1}\Big|_{x_1=0}  \frac{\partial}{\partial x_2} \Big|_{x_2=0} 
\int_{\mathbb{R}^2} 
\frac{1}{2\pi} \log{|x-y|} \, \omega(t, y) \, dy 
\\ \label{eq:L} 
&= 
\frac{1}{\pi} \int_{\mathbb{R}^2} 
\frac{\eta_1(t,x) \eta_2(t,x)}{(\eta_1^2(t,x) + \eta_2^2(t,x))^2} \, \omega_0(x) dx 
\\ \nonumber 
&\geq 
\frac{1}{\pi} \int_{x_1, x_2 \geq 0} 
\frac{\eta_1(t,x) \eta_2(t,x)}{(\eta_1^2(t,x) + \eta_2^2(t,x))^2} \, \omega_0(x) dx 
\\ \nonumber 
&= 
\frac{1}{\pi} \int_{x_1, x_2 \geq 0} 
\left( \frac{\eta_1(t,x)}{\eta_2(t,x)} + \frac{\eta_2(t,x)}{\eta_1(t,x)} \right)^{-1} 
\left( \eta_1^2(t,x) + \eta_2^2(t,x) \right)^{-1} 
\omega_0(x) 
dx. 
\end{align} 

In order to get a suitable lower bound on $\Lambda(t,0)$ we further restrict the integral 
to a sector of the first quadrant defined by 
$$
S 
= 
\left\{ 
x \in \mathbb{R}^2: \frac{1}{2} x_1 \leq x_2 \leq 2 x_1, x_1 \geq 0, x_2 \geq 0 \right\} 
$$  
and observe that for $x \in S$ from \eqref{eq:fi}, \eqref{eq:beta} and \eqref{eq:alfa} we have 
\begin{align} \nonumber 
| \eta_1(t,x)| 
&= 
\left| \tilde{A}_1(t,x) + \tilde{B}_1(t,x) \right| 
\\ \label{eq:fi1} 
&\lesssim 
M x_2 + M^{-1} \sqrt{ x_1^2 + x_2^2} 
\lesssim 
\big( M + M^{-1} \big) x_2 
\\ \nonumber 
&\simeq M x_2
\end{align} 
and similarly 
\begin{align} \label{eq:fifi1} 
| \eta_2(t,x)| 
\lesssim 
\big( M + M^{-1} \big) x_1 
\simeq 
M x_1. 
\end{align} 

On the other hand, for the second component $\eta_2(t,x)$ and $x \in S$, integrating \eqref{eq:e2M} 
and using the sign-preserving property we obtain 
\begin{align*} 
\frac{x_2}{2M} 
\lesssim 
x_2 \int_0^1 e^{\int_0^t \Lambda(\tau, rx) d\tau} dr 
= 
\tilde{A}_2(t,x) 
= 
\eta_2(t,x) - \tilde{B}_2(t,x) 
\leq 
\eta_2(t,x) + | \tilde{B}_2(t,x)| 
\end{align*} 
which by \eqref{eq:beta} gives 
\begin{align*} 
0 \leq x_2 \lesssim M \eta_2(t,x) + tM^3 \sqrt{x_1^2 + x_2^2} 
\end{align*} 
for any $x \in S$ and $0 \leq t \leq T$. 
Therefore, shrinking slightly the time interval once again, if needed, and using $x_1 \leq 2 x_2$ 
we obtain 
\begin{align*} 
0 \leq x_2 \lesssim M \eta_2(t,x) 
\end{align*} 
for any $0 \leq t \leq \min{( T, M^{-3}/2\sqrt{5})}$. Put together with \eqref{eq:fi1} this implies 
\begin{align*} 
| \eta_1(t,x)| 
\lesssim 
M^2 \eta_2(t,x) 
\end{align*} 
and by an analogous argument 
\begin{align*} 
| \eta_2(t,x)| 
\lesssim 
M^2 \eta_1(t,x) 
\end{align*} 
for any $0 \leq t \leq \min{( T, M^{-3}/2\sqrt{5} )}$ and any $x \in S$ which combined give 
\begin{equation} \label{eq:fi12} 
M^{-2} \lesssim \frac{\eta_1(t,x)}{\eta_2(t,x)} \lesssim M^2 
\end{equation} 
for any $0 \leq t \leq \min{( T, M^{-3}/2\sqrt{5} )}$ and any $x \in S$. 

We return to the estimate of $\Lambda(t,0)$ in \eqref{eq:L}. Substituting for $\omega_0$ 
from \eqref{eq:bump} and \eqref{eq:iv} 
and using \eqref{eq:fi1}, \eqref{eq:fifi1} and \eqref{eq:fi12} we obtain 
\begin{align*} 
- \pi \Lambda(t,0) 
&\geq 
\int_{S} 
\left( \frac{\eta_1(t,x)}{\eta_2(t,x)} + \frac{\eta_2(t,x)}{\eta_1(t,x)} \right)^{-1} 
\left( \eta_1^2(t,x) + \eta_2^2(t,x) \right)^{-1} 
\omega_0(x) 
\, dx 
\\ 
& \gtrsim 
M^{-6} N^{-\frac{1}{p}} 
\sum_{k=N_0}^{N_0+N}
\int_{S} 
\frac{\varphi_k(x)}{|x|^2} \, dx 
\\ 
&\hskip -1.5cm 
\simeq 
M^{-6} N^{-\frac{1}{p}} 
\sum_{k=N_0}^{N_0+N} 
2^{(-1 + \frac{2}{p})k} 
\int_{S \cap \mathrm{supp}(\varphi_k)} 
\sum_{\varepsilon_1, \varepsilon_2 = \pm 1} 
\frac{\varepsilon_1 \varepsilon_2 \varphi(2^k x_1 {-} \varepsilon_1, 2^k x_2 {-} \varepsilon_2)}{x_1^2 + x_2^2} 
\, dx_1 dx_2 
\\ 
&\simeq 
M^{-6} N^{-\frac{1}{p}} 
\sum_{k=N_0}^{N_0+N} 
2^{(-1 + \frac{2}{p})k} 
\int_{B_k}  
\frac{\varphi(2^k x_1 {-} 1, 2^k x_2 {-} 1)}{x_1^2 + x_2^2} 
\, dx_1 dx_2 
\end{align*} 
where 
$B_k$ denotes the ball centered at $(2^{-k}, 2^{-k})$ of radius $1/2^{k+2}$, 
see \eqref{eq:suppp}. 
Thus, changing variables we can further bound the above expression from below by 
\begin{align*} 
&\gtrsim 
M^{-6} N^{-\frac{1}{p}} \frac{8}{25} \int_{|x|<\frac{1}{4}} \varphi(x) \, dx 
\sum_{k=N_0}^{N_0+N} 2^{(-1 + \frac{2}{p})k}  
\simeq 
M^{-6} N^{-\frac{1}{p}} \sum_{k=N_0}^{N_0+N} 2^{(-1 + \frac{2}{p})k}. 
\end{align*} 

Recall that from \eqref{eq:e2M} with $t \simeq M^{-3}$ we now have 
\begin{align*} 
\log 2M 
&\gtrsim 
\left| \int_0^{M^{-3}} \Lambda(\tau, 0) \, d\tau \right| 
\gtrsim 
M^{-9} N^{-\frac{1}{p}} \sum_{k=N_0}^{N_0+N} 2^{(-1 + \frac{2}{p})k}
\gtrsim 
M^{-9} \frac{ N^{ 1-\frac{1}{p} } }{ 2^{(1 - \frac{2}{p})(N_0 + N)} }. 
\end{align*} 
Finally, given $M \gg 1$ and $N_0>0$ first choose a large positive integer $N$ so that 
$N^{1/2} \geq 10 M^{10}$ 
and then pick an exponent such that 
$$ 
2< p \leq \frac{2(N_0{+}N)}{N_0{+}N-1}
$$ 
to obtain a desired contradiction.

\section{Proof of Theorem \ref{thm:1}} 
\label{sec:proof} 

Let $2 < p < \infty$, $s=2$ and take $T =1$. Let $\omega_0$ be the initial vorticity defined in 
\eqref{eq:iv} in Section \ref{sec:Lag}. 
As before, let $\omega(t)$ be the corresponding (smooth) solution of \eqref{eq:euler-v}-\eqref{eq:euler-vic} 
and let $\eta(t)$ denote the associated Lagrangian flow of $u = \nabla^\perp\Delta^{-1}\omega$. 

Let $M \gg 1$ be an arbitrary large number. We can consider two cases. 
If there exists $0 < t_0 \leq M^{-3}$ such that 
$\| \omega(t_0)\|_{W^{1,p}} > M^{1/3}$ 
then there is nothing to prove. 
We will therefore assume that 
\begin{equation} \label{eq:assump} 
\|\omega(t_0) \|_{W^{1,p}} \leq M^{1/3}, 
\qquad 
0 \leq t_0 \leq M^{-3}. 
\end{equation} 
By Proposition \ref{prop:Lag} we can then find a point $x^\ast \in \mathbb{R}^2$ such that at least one of 
the entries $\partial\eta^i/\partial x_j$ of the Jacobi matrix (for example, the $i{=}j{=}2$ entry) 
satisfies $| \partial_2\eta_2 (t_0,x^\ast)| > M$. 
Therefore, by continuity, there is a $\delta >0$ such that 
\begin{align} \label{eq:M} 
\left| \frac{\partial \eta_2}{\partial x_2} (t_0,x) \right| > M 
\qquad 
\text{for all} 
\quad 
|x-x^\ast| < \delta. 
\end{align} 
Let $0 \leq \rho \leq 1$ be a smooth bump function on $\mathbb{R}^2$ 
with $\mathrm{supp}\, \rho \subset B(0,2)$ and such that $\rho \equiv 1$ on $B(0,1)$. 
For any $k \in \mathbb{Z}_+$ and $\lambda >0$ define 
\begin{equation} \label{eq:beta-pert} 
\beta_{k,\lambda}(x) 
= 
\frac{\lambda^{-1 + \frac{2}{p}}}{\sqrt{k}} 
\sum_{\varepsilon_1, \varepsilon_2 = \pm 1} 
\varepsilon_1 \varepsilon_2  \rho(\lambda(x-x^\ast_\epsilon)) \sin{kx_1} 
\end{equation} 
where 
$x^\ast_\epsilon = (\varepsilon_1 x^\ast_1, \varepsilon_2 x^\ast_2)$ 
and 
$x^\ast = (x^\ast_1, x^\ast_2)$. 
Observe that $\beta_{k,\lambda}$ are \textit{smooth} functions with compact support 
in $\mathbb{R}^2$ 
\begin{equation} \label{eq:supp} 
\mathrm{supp}\, (\beta_{k,\lambda}) 
\subset 
\bigcup_{\varepsilon_1, \varepsilon_2} 
B\big( x^\ast_\varepsilon, 2/\lambda \big). 
\end{equation} 

In the sequel we will need two technical lemmas. 
\begin{lem} \label{lem:rem} 
For any $k \in \mathbb{Z}_+$ and $\lambda >0$ we have 
\begin{enumerate} 
\item[1.] 
$
\| \partial_j \Delta^{-1} \beta_{k,\lambda} \|_\infty 
\lesssim 
k^{-1/2} \lambda^{-1 + \frac{2}{p}} \| \rho\|_\infty 
$ 
\vskip 0.1cm 
\item[2.] 
$
\| \partial_i \partial_j \Delta^{-1} \beta_{k,\lambda} \|_\infty 
\lesssim 
k^{-1/2} \lambda^{-1 + \frac{2}{p}} \| \hat{\rho}\|_{L^1} 
$ 
\vskip 0.15cm 
\item[3.] 
$
\| \beta_{k,\lambda} \|_{W^{1,p}} 
\lesssim 
\big( k^{1/2}\lambda^{-1} + k^{-1/2} + k^{-1}\lambda^{-1} \big) \|\rho\|_{W^{1,p}} 
$ 
\end{enumerate} 
where $i, j =1, 2$ and the bounds depend on $\rho$ and $x^\ast$. 
\end{lem} 
\begin{proof} 
Observe that for any $x \in \mathbb{R}^2$ we have 
$$ 
\mathcal{F}^{-1}\mathcal{F}\big(\partial_j\Delta^{-\frac{1}{2}} \beta_{k,\lambda} \big)(x) 
= 
K \ast \beta_{k,\lambda}(x) 
= 
\int_{\mathbb{R}^2} K(x-y) \beta_{k,\lambda}(y) \, dy 
$$ 
where $\mathcal{F}, \mathcal{F}^{-1}$ denote the Fourier transform and its inverse, respectively, 
and where the kernel $K$ is homogeneous of degree $-1$, that is, 
$$
K(tx) = t^{-1}K(x) 
\qquad 
\text{for any} 
\quad 
t>0, \, x \neq 0. 
$$ 
In particular $K \in L^1_{\mathrm{loc}}(\mathbb{R}^2)$ and consequently we can obtain 
an $L^\infty$ bound for the integral operator by a direct calculation using the fact that 
$\beta_{k,\lambda}$ has compact support. 
Namely, for any $\epsilon>0$ we have 
\begin{align*} 
\left| \partial_j \Delta^{-1/2} \beta_{k,\lambda} (x) \right| 
&= 
\left| \left( \int_{|x-y|<\epsilon} + \int_{|x-y|>\epsilon} \right) 
K(x-y) \beta_{k,\lambda}(y) \, dy \right| 
\\ 
&\lesssim 
\| \beta_{k,\lambda}\|_\infty \int_{|x-y|<\epsilon} \frac{dy}{|x-y|} 
+ 
\frac{1}{\epsilon} \int_{\mathrm{supp}\, \beta_{k,\lambda}} |\beta_{k,\lambda}(y)| \, dy 
\\ 
&\lesssim 
\| \beta_{k,\lambda}\|_\infty \int_{|y|<\epsilon} |y|^{-1} dy 
+ 
\epsilon^{-1} \|\beta_{k,\lambda}\|_\infty \mu(\mathrm{supp}\, \beta_{k,\lambda}). 
\end{align*} 
Note that by \eqref{eq:supp} if $\lambda>1$ then $\mu(\mathrm{supp}\, \beta_{k,\lambda}) \leq 16\pi^2$ 
so that from \eqref{eq:beta-pert} we get 
$$
\| \partial_j \Delta^{-1} \beta_{k,\lambda} \|_\infty 
\lesssim 
C_{\epsilon,x^*} \|\beta_{k,\lambda}\|_\infty 
\lesssim 
C_{\epsilon, x^*} \| \rho \|_\infty k^{-1/2} \lambda^{-1 + \frac{2}{p}} 
$$ 
where $C_{\epsilon, x^*}>0$ is a constant depending only on $\epsilon >0$ and $x^*$. 

For the second assertion let $\xi_{\pm} = (\xi_1 \pm k/2\pi, \xi_2)$ and first compute the Fourier transform 
$$ 
\hat{\beta}_{k,\lambda}(\xi) 
= 
\frac{\lambda^{-1+\frac{2}{p}}}{\sqrt{k}} 
\sum_{\varepsilon_1, \varepsilon_2} 
\frac{\varepsilon_1 \varepsilon_2}{2i} 
\left( 
e^{-2\pi i \langle \xi_{-}, x^*_{\varepsilon} \rangle} 
\frac{1}{\lambda^2} \hat{\rho}\Big( \frac{\xi_{-}}{\lambda} \Big) 
- 
e^{-2\pi i \langle \xi_{+}, x^*_{\varepsilon} \rangle} 
\frac{1}{\lambda^2} \hat{\rho}\Big( \frac{\xi_{+}}{\lambda} \Big) 
\right). 
$$ 
Next, we estimate 
\begin{align*} 
\big| \partial_i \partial_j \Delta^{-1} \beta_{k,\lambda}(x) \big| 
&\simeq 
\Big| \mathcal{F}^{-1}\mathcal{F} \big( \partial_i \partial_j \Delta^{-1} \beta_{k,\lambda} \big) (x) \Big|
\lesssim 
\| \hat{\beta}_{k,\lambda} \|_{L^1} 
\\ 
& \lesssim 
k^{-1/2} \lambda^{-1+\frac{2}{p}} \sum_{A=\pm} 
\int_{\mathbb{R}^2} \frac{1}{\lambda^2} \Big| 
\hat{\rho} \Big(\frac{\xi_A}{\lambda} \Big) 
\Big| d\xi 
\\ 
&\simeq 
k^{-1/2} \lambda^{-1+ \frac{2}{p}} \| \hat{\rho}\|_{L^1}. 
\end{align*} 

Finally, using once again the triangle inequality and the change of variables formula we compute 
\begin{align*} 
\Big\| \frac{\partial\beta_{k,\lambda}}{\partial x_1} \Big\|_{L^p} 
&\lesssim 
\frac{1}{\sqrt{k}} \Big\| 
\lambda^{2/p} \sum_{\varepsilon_1,\varepsilon_2} 
\varepsilon_1 \varepsilon_2 \frac{\partial \rho}{\partial x_1}\big( \lambda(\cdot - x^*_\varepsilon) \big) 
\Big\|_{L^p} 
+ 
\frac{\sqrt{k}}{\lambda} \Big\| 
\lambda^{2/p} \sum_{\varepsilon_1,\varepsilon_2} 
\varepsilon_1 \varepsilon_2 \rho\big(\lambda(\cdot-x^*_\varepsilon)\big) 
\Big\|_{L^p} 
\\ 
&\simeq 
\frac{1}{\sqrt{k}} \left( \int_{\mathbb{R}^2} 
\Big| \sum_{\varepsilon_1, \varepsilon_2} 
\varepsilon_1\varepsilon_2 \frac{\partial\rho}{\partial x_1}(x) \Big|^p 
dx \right)^{1/p} 
+ 
\frac{\sqrt{k}}{\lambda} \left( \int_{\mathbb{R}^2} 
\Big| \sum_{\varepsilon_1, \varepsilon_2} 
\varepsilon_1\varepsilon_2 \rho(x) \Big|^p 
dx \right)^{1/p} 
\\ 
&\lesssim 
k^{-1/2} \Big\| \frac{\partial \rho}{\partial x_1}\Big\|_{L^p} 
+ 
k^{1/2}\lambda^{-1} \| \rho\|_{L^p}. 
\end{align*} 
Similarly, we find 
\begin{align*} 
\Big\| \frac{\partial\beta_{k,\lambda}}{\partial x_2} \Big\|_{L^p} 
\lesssim 
k^{-1/2} \Big\| \frac{\partial \rho}{\partial x_2} \Big\|_{L^p} 
\end{align*} 
and 
\begin{align*} 
\| \beta_{k,\lambda}\|_{L^p} 
\lesssim 
\frac{\lambda^{-1}}{\sqrt{k}} \left( \int_{\mathbb{R}^2} 
\Big| \lambda^{2/p} \sum_{\varepsilon_1, \varepsilon_2} 
\varepsilon_1\varepsilon_2 \rho\big( \lambda(x-x^\ast_\varepsilon) \big) \Big|^p dx \right)^{1/p} 
\lesssim 
k^{-1/2} \lambda^{-1} \| \rho\|_{L^p} 
\end{align*} 
which combined yield the lemma. 
\end{proof} 

Observe that choosing 
\begin{equation} \label{eq:kln} 
k= \lambda^2 
\quad \text{and} \quad 
\lambda = 3n, 
\quad 
n \in \mathbb{Z}_+ 
\end{equation} 
and letting 
$\beta_n = \beta_{k,\lambda}$ 
in Lemma \ref{lem:rem} we immediately obtain 
\begin{align} \label{eq:n1} 
&\| \partial_j \Delta^{-1} \beta_n \|_\infty  \rightarrow 0 
\quad 
\text{and} 
\quad 
\| \partial_j \nabla^\perp \Delta^{-1} \beta_n \|_\infty \rightarrow 0 
\quad 
\text{as} 
\; 
n \to \infty 
\end{align} 
as well as 
\begin{align} \label{eq:n2} 
&\| \beta_n\|_{W^{1,p}} \lesssim \|\rho\|_{W^{1,p}} < \infty 
\quad 
\text{for any} 
\; 
n \in \mathbb{Z}_+. 
\end{align} 
We will also need the following 
\begin{lem} \label{lem:remrem} 
With $k, \lambda$ and $n$ as in \eqref{eq:kln} we have 
\begin{enumerate} 
\item[1.] 
$
\| \partial_2 \beta_{k,\lambda} \partial_1 \eta_2(t_0) \|_{L^p} 
\lesssim 
k^{-1/2} \| \partial_1\eta_2(t_0)\|_{L^\infty(\cup B(x^*_\varepsilon,2))} 
\xrightarrow[n \to \infty]{} 0 
$ 
\vskip .1cm
\item[2.] 
$
\| \partial_1\beta_{k,\lambda} \partial_2\eta_2(t_0) \|_{L^p} 
\gtrsim 
M k^{1/2} \lambda^{-1} - k^{-1/2} \| \partial_2\eta_2(t_0)\|_{L^\infty(B(x^\ast,2))} 
\xrightarrow[n \to \infty]{} M 
$ 
\end{enumerate} 
where the constants in both estimates depend only on $\rho$, $x^\ast$ and $t_0$. 
\end{lem} 
\begin{proof} 
In the first case from \eqref{eq:beta-pert} we have 
\begin{align*} 
\bigg( \int_{\mathbb{R}^2} \Big| 
&\frac{\partial \beta_{k,\lambda}}{\partial x_2}(x) \frac{\partial \eta_2}{\partial x_1}(t_0, x) 
\Big|^p dx \bigg)^{1/p} 
\\ 
&= 
\left( \int_{ \cup B( x^\ast_\varepsilon, \frac{2}{\lambda}) } \bigg| 
\frac{1}{\sqrt{k}} \lambda^{\frac{2}{p}} \sum_{\varepsilon_1, \varepsilon_2} 
\varepsilon_1 \varepsilon_2 
\frac{\partial \rho}{\partial x_2}( \lambda(x-x^*_\varepsilon)) \frac{\partial\eta_2}{\partial x_1}(t_0,x) \sin{kx_1} 
\bigg|^p dx \right)^{1/p} 
\\ 
&\leq 
\frac{1}{\sqrt{k}} \sum_{\varepsilon_1, \varepsilon_2} \left( 
\int_{ \cup B( x^\ast_\varepsilon, \frac{2}{\lambda}) } 
\lambda^2 \Big| \frac{\partial\rho}{\partial x_2}(\lambda(x-x^*_\varepsilon)) \Big|^p dx 
\right)^{1/p} 
\sup_{ \cup B( x^\ast_\varepsilon, 2) } \Big| \frac{\partial\eta_2}{\partial x_1}(t_0,x) \Big| 
\\ 
&= 
4 k^{-1/2} \| \partial_1\eta_2(t_0)\|_{L^\infty(\cup B(x^*_\varepsilon,2))} 
\left( \int_{B(0,2)} \Big| 
\frac{\partial \rho}{\partial x_2}(x) \Big|^p dx \right)^{1/p} 
\end{align*} 
which gives the desired estimate. 

The second case is slightly more cumbersome. We have 
\begin{align*} 
\bigg( \int_{\mathbb{R}^2} \Big| 
&\frac{\partial \beta_{k,\lambda}}{\partial x_1}(x) \frac{\partial \eta_2}{\partial x_2}(t_0, x) 
\Big|^p dx \bigg)^{1/p} 
= 
\\ 
&= 
\Bigg( \int_{ \mathbb{R}^2 } \bigg| 
\frac{1}{\sqrt{k}} \lambda^{\frac{2}{p}-1} \sum_{\varepsilon_1, \varepsilon_2} 
\varepsilon_1 \varepsilon_2 
\Big( 
k \rho(\lambda(x-x^*_\varepsilon)) \cos{kx_1} 
\, + 
\\ 
& \hskip 3cm + 
\lambda \frac{\partial\rho}{\partial x_1} (\lambda(x - x^*_\varepsilon)) \sin{kx_1} 
\Big) 
\frac{\partial \eta_2}{\partial x_2}(t_0,x) 
\bigg|^p dx \Bigg)^{1/p} 
\\ 
&\geq 
\Bigg( 
\int_{B(x^*,\frac{2}{\lambda}) \cap B(x^*,\delta)} 
\bigg| 
\sqrt{k} \lambda^{-1 + \frac{2}{p}} \cos{kx_1} \rho(\lambda(x-x^*)) \frac{\partial\eta_2}{\partial x_2}(t_0,x) 
\, + 
\\ 
&\hskip 3cm + 
\frac{1}{\sqrt{k}} \lambda^{\frac{2}{p}} \sin{kx_1} \frac{\partial\rho}{\partial x_1}(\lambda(x-x^*)) 
\frac{\partial\eta_2}{\partial x_2}(t_0,x) 
\bigg|^p 
dx 
\Bigg)^{1/p}. 
\end{align*} 
Taking $\lambda$ large enough so that $2/\lambda < \delta$ and using \eqref{eq:M} and 
the triangle inequality 
we can further estimate the above integral from below by 
\begin{align*} 
M \sqrt{k} \lambda^{-1} 
\bigg( 
\int_{B(x^*,\frac{2}{\lambda})} &\lambda^2 
\big| \cos{kx_1} \rho(\lambda(x-x^*)) \big|^p 
dx \bigg)^{1/p} 
- 
\\ 
&- 
\frac{1}{\sqrt{k}} \bigg( 
\int_{B(x^*,\frac{2}{\lambda})} \lambda^2 
\Big| \sin{kx_1} \frac{\partial\rho}{\partial x_1}(\lambda(x-x^*)) \frac{\partial\eta_2}{\partial x_2}(t_0,x)\Big|^p 
dx \bigg)^{1/p} 
\\ 
&\hskip - 2.2cm \geq 
M \sqrt{k} \lambda^{-1} 
\bigg( 
\int_{B(0,1)} \big| \cos{(k\lambda^{-1}x_1 + kx^*)} \big|^p 
dx \bigg)^{1/p} 
- 
\\ 
&- 
\frac{1}{\sqrt{k}} \| \partial_1\rho \|_{L^p(B(0,2))} \|\partial_2\eta_2(t_0)\|_{L^\infty(B(x^*,2))} 
\end{align*} 
where in the last step we changed variables $x \to \lambda x - \lambda x^*$ and used the fact 
that $\rho \equiv 1$ on the unit ball $B(0,1)$ 
by construction. 
It now suffices to observe that the integral term can bounded from below for the choices of 
the parameters made in \eqref{eq:kln}. In fact, since $p>2$ we have 
\begin{align*} 
\bigg( 
\int_{B(0,1)} \big| \cos(k\lambda^{-1}x_1 {+} kx^*_1) \big|^p 
dx \bigg)^{1/p} 
&\geq 
\left( \int_{-\pi/6}^{\pi/6} \int_{-\pi/6}^{\pi/6} 
\cos^2 (\lambda x {+} \lambda^2 x^*_1) \, dx dy \right)^{1/2} 
\\ 
&= 
\frac{\pi}{3\sqrt{2}} 
\end{align*} 
by a straightforward calculation. 
\end{proof} 

For each $n \in \mathbb{Z}_+$ consider the following sequence of 
(smooth) initial vorticities with compact support 
\begin{equation} \label{eq:omega-seq} 
\omega_{0,n}(x) = \omega_0(x) + \beta_n(x). 
\end{equation} 
Note that from \eqref{eq:n2} and Lemma \ref{lem:omega-0} it follows that $\omega_{0,n}$ belongs to 
$W^{1,p}$ for any $n$ in $\mathbb{Z}_+$. 
Let $\omega_n(t)$ be the corresponding (smooth) solutions of the vorticity equations 
\eqref{eq:euler-v}-\eqref{eq:euler-vic}. 

We now come to the crucial step in our construction. 
For each $n \in \mathbb{Z}_+$ let $\eta_n(t)$ be the flow of volume-preserving diffeomorphisms 
of the associated velocity fields $u_n = \nabla^\perp\Delta^{-1}\omega_n$ 
as in \eqref{eq:flow}-\eqref{eq:flow-ic}. 
Assume that the data-to-solution map for the Euler equations is continuous from bounded sets 
in $C^1(\mathbb{R}^2)$ to $C([0,1], C^1(\mathbb{R}^2))$. 
It then follows from \eqref{eq:omega-seq} and \eqref{eq:n1} that 
\begin{equation} \label{eq:A} 
\sup_{0 \leq t \leq 1}\| D\Delta^{-1} \nabla^\perp ( \omega_n(t) - \omega(t) ) \|_\infty 
\longrightarrow 0 
\quad 
\text{as} 
\quad 
n \to \infty 
\end{equation} 
as well as 
\begin{equation} \label{eq:T} 
\sup_{0 \leq t \leq 1} \| \Delta^{-1} \nabla^\perp(\omega_n(t) - \omega(t)) \|_\infty 
\longrightarrow 0 
\quad 
\text{as} 
\quad 
n \to 
\infty. 
\end{equation} 
(cf. also Thm. 2.12; inequality (2.21) of \cite{TTY}). 
Applying the comparison Lemma \ref{lem:comp} we then find 
\begin{align} \label{eq:LL} 
\sup_{0\leq t \leq 1} \big( 
\| \eta_n(t) - \eta(t) \|_\infty + \| D\eta_n(t) - D\eta(t) \|_\infty 
\big) 
= 
\theta_n 
\longrightarrow 0 
\quad 
\text{as} 
\; 
n \to \infty 
\end{align} 
where $\eta(t)$ is the flow of the velocity field $u=\nabla^\perp\Delta^{-1}\omega$ 
with the initial vorticity $\omega_0$ given by \eqref{eq:iv} as in Proposition \ref{prop:Lag}. 

Using conservation of vorticity, formula \eqref{eq:SGr} of Remark \ref{rem:SGr} and 
the invariance of the $L^p$ norms under volume-preserving Lagrangian flows $\eta_n(t)$ 
(change of variables) we have 
\begin{align} \nonumber 
\qquad 
\| \omega_n(t_0) \|_{W^{1,p}} 
&\geq 
\| \nabla( \omega_{0,n}\circ\eta_n^{-1}(t_0)) \|_{L^p} 
\simeq 
\\  \nonumber 
&\hskip -2.8cm 
\| d\omega_{0,n}{\circ}\eta_n^{-1}(t_0) (\nabla^\perp\eta_{n,2}(t_0) {\circ}\eta_n^{-1}(t_0)) \|_{L^p} 
{+} 
\| d\omega_{0,n}{\circ}\eta_n^{-1}(t_0) (\nabla^\perp\eta_{n,1}(t_0){\circ}\eta_n^{-1}(t_0)) \|_{L^p} 
\\  \label{eq:BB} 
&\simeq 
\| d\omega_{0,n} (\nabla^\perp\eta_{n,2}(t_0)) \|_{L^p} 
+ 
\| d\omega_{0,n} (\nabla^\perp\eta_{n,1}(t_0)) \|_{L^p} 
\\ \nonumber 
&\gtrsim 
\| d\omega_{0,n} ( \nabla^\perp\eta_{n,2}(t_0)) \|_{L^p}. 
\end{align} 
Since from the comparison estimate \eqref{eq:LL} we have 
$$ 
\| d\omega_{0,n}( \nabla^\perp\eta_2 - \nabla^\perp\eta_{n,2})(t_0)\|_{L^p} 
\lesssim 
\| D( \eta_2 - \eta_{n,2} )(t_0)\|_\infty \|\nabla\omega_{0,n}\|_{L^p} 
\leq 
\theta_n \|\nabla\omega_{0,n}\|_{L^p} 
$$ 
applying the triangle inequality and \eqref{eq:omega-seq} we can further bound 
the right side of the expression in \eqref{eq:BB} below by 
\begin{align} \label{eq:MT}  
&\| d\omega_{0,n} (\nabla^\perp\eta_2(t_0)) \|_{L^p} 
- 
\theta_n \| \nabla\omega_{0,n}\|_{L^p} 
\\  \nonumber 
&\hskip 1cm \gtrsim 
\| d\beta_n (\nabla^\perp\eta_2(t_0)) \|_{L^p} 
- 
\| d\omega_0(\nabla^\perp\eta_2(t_0))\|_{L^p} 
- 
\theta_n \| \nabla\omega_{0,n}\|_{L^p}. 
\end{align} 
Observe that by the assumption \eqref{eq:assump} we can bound the middle term 
on the right side of \eqref{eq:MT} as in \eqref{eq:BB} above by 
\begin{align} \nonumber 
\| d\omega_0( \nabla^\perp\eta_2(t_0)) \|_{L^p} 
&\leq 
\| \nabla\omega_0\circ\eta^{-1}(t_0) \cdot D\eta^{-1}(t_0) \|_{L^p} 
\\  \label{eq:WW} 
&\simeq 
\| \nabla( \omega_0\circ\eta^{-1} (t_0) )\|_{L^p} 
\leq 
\| \omega(t_0) \|_{W^{1,p}} 
\leq M^{1/3}. 
\end{align} 

It therefore remains to find a lower bound on the $\beta$-term in \eqref{eq:MT}. 
This however follows from the the two estimates in Lemma \ref{lem:remrem}. 
Namely, we have 
\begin{align}  
\| d\beta_n( \nabla^\perp\eta_2(t_0)) \|_{L^p} 
&= 
\big\| -\partial_1\beta_n \partial_2\eta_2(t_0) + \partial_2\beta_n \partial_1\eta_2(t_0) \big\|_{L^p} 
\nonumber 
\\ \nonumber 
&\geq 
\| \partial_1\beta_n \partial_2\eta_2(t_0)\|_{L^p} 
- 
\| \partial_2 \beta_n \partial_1\eta_2(t_0)\|_{L^p} 
\\   \label{eq:bBeta} 
&\gtrsim 
M 
- 
\frac{1}{n} \| \partial_2\eta_2(t_0)\|_{L^\infty(B(x^*,2))} 
- 
\frac{1}{n} \| \partial_1\eta_2(t_0)\|_{L^\infty(\cup B(x^*_\varepsilon,2))} 
\\  \nonumber 
&\gtrsim 
M 
\end{align} 
provided that $n$ is sufficiently large. 
Combining \eqref{eq:BB}, \eqref{eq:MT}, \eqref{eq:WW} and \eqref{eq:bBeta} 
completes the proof of Theorem \ref{thm:1}.

\bibliographystyle{amsplain}

\end{document}